\documentclass{ReAnal-axSIH}

\newtheorem{theorem}{Theorem}
\newtheorem{proposition}[theorem]{Proposition}

\theoremstyle{definition}

\newtheorem{definitions}[theorem]{\font\SweD =cmssi10\SweD Definitions\bf}
\newtheorem{remark}[theorem]{\font\SweD =cmssi10\SweD Remark\bf}

\newcommand\bbC{\mathbb C}
\renewcommand\math[1]{\hbox{$\kern0.4mm#1\kern0.4mm$}}
\newcommand\mathss[3]{\hbox{\kern0.17mm\kern.#1mm$#3$\kern.#2mm\kern0.17mm}}
\def\rbrakf{\kern.9mm]\kern.2mm\lower.8mm\hbox{\font\SweD =cmr5\SweD f}\kern.2mm} 
\def\bmii#1#2{\hbox{\font\=cmmib#1\#2}}
\newcommand\bosy{\boldsymbol}
\def\expnota^#1]_#2{\,^{#1\,]{_{}}_{\roman{#2}}}}

\def\bbNo{\mathbb N\kern.15mm\lower.65mm\hbox{\font\SweD =cmr6\SweD 0}\kern.1mm} 
\def\sbbNo{{\mathbb N\kern.07mm\lower.45mm\hbox{\font\SweD =cmr5\SweD 0}\kern.1mm}} 
\def\bbR{\mathbb R} 
\newcommand\ssbb[3]{\hspace{.#1mm}\mathbb#3\hspace{.#2mm}} 
\def\lbb#1_#2{\mathbb#1\kern.2mm\raise.52mm\hbox{$_{_{#2}}$}} 
\def\rbb#1^#2{\mathbb#1\kern.2mm\lower.55mm\hbox{$^{^{#2}}$}} 
\def\Rlplus{\mathbb R\kern.2mm\lower.33mm\hbox{\font\SweD =cmr5\SweD +}} 
\def\fbbR{\raise1.2mm\hbox{\font\SweD =cmr5\SweD f}\kern.3mm\mathbb R} 
\def\tfbbR{\raise1.2mm\hbox{\font\SweD =cmr5\SweD tf}\kern.3mm\mathbb R} 
\def\stfbbR{{\raise.5mm\hbox{\font\SweD =cmr5\SweD t\kern-.1mmf}\kern.3mm\mathbb R}} 
\newcommand\tvbbR[1]{\raise1.35mm\hbox{\font\SweD=cmr5\SweD t}\raise1.65mm\hbox{\font\SweD=cmss5\SweD v}\kern.2mm\mathbb R\kern.3mm\kern.#1mm} 
\newcommand\smbbR[1]{\raise1.6mm\hbox{\font\SweD=cmr5\SweD sm}\kern.1mm\mathbb R\kern.3mm\kern.#1mm} 
\def\fbbC{\raise1.23mm\hbox{\font\SweD =cmr5\SweD f}\kern.1mm\mathbb C} 
\def\tfbbC{\raise1.23mm\hbox{\font\SweD =cmr5\SweD tf}\kern.1mm\mathbb C} 
\def\stfbbC{{\raise.5mm\hbox{\font\SweD =cmr5\SweD t\kern-.1mmf}\kern.3mm\mathbb C}} 
\def\tfbbH{\raise1.2mm\hbox{\font\SweD =cmr5\SweD tf}\kern.3mm\mathbb H} 
\def\taubb_#1{\tau\kern-.15mm\lower.7mm\hbox{\font\SweD =msbm5\SweD #1}\kern.3mm} 
\def\nsTbb_#1{\hbox{\font\SweD =eusm9\SweD T}\lower.7mm\hbox{\font\SweD =msbm5\SweD #1}\kern.3mm} 
\newcommand\exbbR{\kern.2mm\overline{\kern-.2mm\mathbb R\kern-.4mm}\kern.35mm} 
\newcommand\TopexR{\hbox{\font\å=eusm9\åT}\lower.27mm\hbox{$_{^{\overline{\mathbb R\kern-.4mm}}}$}\kern.388mm} 
\def\bartau_bb#1{\bar\tau\kern-.15mm\lower.7mm\hbox{\font\SweD =msbm5\SweD #1}\kern.3mm} 
\def\cinfty{\raise1.35mm\hbox{\font\SweD =cmr5\SweD c}\kern-.15mm\infty} 
\def\inftyyplus{{\vphantom{p_{p_p}}\infty\RHB{.5}{\fiveroman+}}} 
\newcommand\iimag{{\hbox{\kern-.3mm\font\fff=cmr7\fff\char'136\kern-1.05mm\char'020\kern.07mm}}} 
\newcommand\thepi{{\raise1.4mm\hbox{$.$}\kern-.4mm\pi}} 
\newcommand\tthepi{{\raise.95mm\hbox{\font\fff=cmr7\fff\char'056}\kern-.43mm\pi}} 
\newcommand\rmdss[2]{\kern1mm\kern.#1mm\roman d\kern.#2mm\,}
\newcommand\sigmalg[1]{\hbox{\font\SweD =cmmi12\SweD \char'033}\kern-.2mm\lower.77mm\hbox{\font\SweD =cmr6\SweD a}\lower.8mm\hbox{\font\SweD =cmr5\SweD l}\lower.45mm\hbox{\font\SweD =cmr5\SweD g}\kern.#1mm\kern.4mm} 
\newcommand\sigmAlg[1]{\hbox{\font\SweD =cmmi12\SweD \char'033}\kern-.2mm\lower.8mm\hbox{\font\SweD =cmr5\SweD Al}\lower.45mm\hbox{\font\SweD =cmr5\SweD g}\kern.#1mm\kern.4mm} 
\def\Lebmea^#1{\kern.25mm\mu_{\kern.3mm\hbox{\font\SweD =cmr5\SweD Leb}}^{\vphantom n\kern.5mm{#1}}} 
\def\Lebmef^#1{\kern.25mm\mathfrak m_{\kern.3mm\hbox{\font\SweD =cmr5\SweD Leb}}^{\vphantom n\kern.5mm{#1}}} 
\def\LeBmef^#1{\kern.25mm\mathfrak m_{\kern.3mm\hbox{\font\SweD =cmr5\SweD LeB}}^{\vphantom n\kern.5mm{#1}}} 
\newcommand\openIval[1]{\null\kern.35mm]\kern.7mm#1\kern.8mm[\kern.35mm\null} 

\newcommand\sefRC{\{\kern0.15mm\raise1.2mm\hbox{\font\SweD =cmr5\SweD f}\kern.3mm\mathbb R\,,\kern-0.3mm\raise1.23mm\hbox{\font\SweD =cmr5\SweD f}\kern.1mm\mathbb C\,\}} 
\def\setRC{\{\kern0.15mm\raise1.2mm\hbox{\font\SweD =cmr5\SweD tf}\kern.3mm\mathbb R\,,\kern-0.3mm\raise1.23mm\hbox{\font\SweD =cmr5\SweD tf}\kern.1mm\mathbb C\,\}} 
\def\Reit#1{_{\lower.2mm\hbox{\kern.2mm\hskip.#1mm\font\SweD =msbm5\SweD R\kern.15mm\font\SweD =cmr5\SweD t}\kern.15mm}} 
\def\Reif#1{_{\lower.2mm\hbox{\kern.#1mm\font\SweD =msbm5\SweD R\kern.2mm\font\SweD =cmr5\SweD i\kern-.3mm f}\kern.15mm}} 
\def\Ceit#1{_{\lower.2mm\hbox{\hskip.#1mm\font\SweD =msbm5\SweD C\kern.15mm\font\SweD =cmr5\SweD t}\kern.15mm}} 
\def\Ceif#1{_{\lower.2mm\hbox{\kern.#1dd\font\SweD =msbm5\SweD C\kern.2mm\font\SweD =cmr5\SweD i\kern-.3mm f}\kern.15mm}} 
\def\TVSps#1(#2){\roman{TVS}\kern0.6mm(\kern.#1pt\boldsymbol#2\kern0.37mm)} 
\def\LCSps#1(#2){\roman{LCS}\kern0.6mm(\kern.#1pt\boldsymbol#2\kern0.37mm)} 
\def\BaSps#1(#2){\roman{BaS}\kern0.6mm(\kern.#1pt\boldsymbol#2\kern0.37mm)} 
\def\HilbSps#1(#2){\roman H\kern.15mm\lower.7mm\hbox{\font\SweD =cmr5\SweD ilb}\kern0.4mm\roman S\kern.6mm(\kern.#1pt\boldsymbol#2\kern0.37mm)} 
\def\TVS{\roman{TVS}\kern0.4mm}%
\def\LCS{\roman{LCS}\kern0.4mm}%
\def\BaS{\roman{BaS}\kern0.4mm}%
\def\HilbLCS{\roman{H\kern.15mm\lower.7mm\hbox{\font\SweD =cmr5\SweD ilb}\kern.25mmLCS}\kern0.4mm} 
\def\Nbh{\Cal N_{\font\SweD =cmmi6\lower.15mm\hbox{\kern.1mm\SweD bh\kern.15mm}}}
\newcommand\neiBoo{\Cal N\kern-.4mm\lower.7mm\hbox{\font\SweD=cmr7\SweD o}\kern.8mm} 
\newcommand\bouSet{\Cal B\kern.2mm\lower.7mm\hbox{\font\SweD=cmr7\SweD s}\kern.8mm} 
\newcommand\semiNor{\Cal S\kern.2mm\lower.75mm\hbox{\font\SweD=cmr5\SweD N}\kern.8mm} 
\newcommand\bouNor{\Cal B\kern.2mm\lower.75mm\hbox{\font\SweD=cmr5\SweD N}\kern.8mm} 
\newcommand\vsquotient{\lower.85mm\hbox{\kern-.6mm\font\x=cmr5\x vs\kern.85mm}} 
\newcommand\tvsquotient{\lower.85mm\hbox{\kern-.35mm\font\x=cmr5\x tvs\kern.85mm}} 

\def\SemiNor{\Cal S_{_N}\kern0.15mm}
\def\dimHa{{\rm dim_{_{\kern.2mm Ha}}}}

\def\Cinfty{C\kern.6mm\raise.15mm\hbox{$^\infty$}\kern.07mm} 
\def\roman#1{{\rm#1}}
\def\Univ{\hbox{\font\SweD =cmssbx10\SweD U}{}} 
\def\Pows{\mathcal P\kern-.7mm\lower.15mm\hbox{$_s$}\kern.3mm} 
\def\lei{      {}_{ {}^{\,\downarrow\text{\hskip-2.1mm}       }  }  \cap       }
\def\lei{\hbox{\kern.45mm$_{^\downarrow}\kern-1.280mm\cap\kern.85mm$}}
\def\leip{\hbox{\kern.6mm$_{^\downarrow}\kern-1.280mm\cap\kern1mm$}}
\def\leiss#1#2{\hbox{\kern.5mm\kern.#1mm$_{^\downarrow}\kern-1.280mm\cap\kern.9mm\kern.#2mm$}} 
\def\supp{\roman{supp\kern1mm}} 
\def\vfsupp{\text{\,-\,\raise1.4mm\hbox{\font\SweD =cmr5\SweD f}supp\kern1mm}} 

\newcommand\inve{\kern.27mm\raise1.4mm\hbox{\font\SweD =cmr5\SweD -\kern-.28mm-}\raise1.35mm\hbox{\font\SweD =cmss5\SweD i\kern.18mm}} 
\newcommand\invss[2]{\kern.27mm\kern.#1mm\raise1.4mm\hbox{\font\SweD =cmr5\SweD -\kern-.28mm-}\raise1.35mm\hbox{\font\SweD =cmss5\SweD i\kern.#2mm\kern.18mm}} 
\newcommand\invxs[2]{\kern.27mm\kern.#1mm\raise1.4mm\hbox{\font\SweD =cmr5\SweD -\kern-.28mm-in\kern.#2mm\kern.18mm}} 
\newcommand\fvalss[2]{\hbox{\kern.3mm\kern.#1mm\font\SweD =cmr10\SweD \char'022\kern-.1mm}\kern.#2mm} 
\def\fvalue{\hbox{\kern.2mm\font\SweD =cmr10\SweD \char'022\kern-.2mm}} 
\def\ffvalue{\hbox{\kern.2mm\font\SweD =cmr7\SweD \char'022\kern-.2mm}} 
\def\image{\hbox{\font\SweD =cmr10\SweD \char'022\kern-1mm\char'022}} 
\def\iimage{\hbox{\font\SweD =cmr7\SweD \kern.3mm\char'022\kern-.7mm\char'022\kern-.3mm}} 
\def\images{\hbox{\font\SweD =cmr10\SweD \char'022\kern-1mm\char'022\kern-1mm\char'022}} 
\def\weco{\kern.15mm\hbox{\font\SweD =cmtt10\SweD \char'054}\kern.4mm} 
\def\cdotn{\kern-.2mm\cdot\kern-.2mm} 
\def\setminusn{\kern-.2mm\setminus\kern-.2mm} 
\def\capss#1#2{\kern.1mm\kern.#1mm\cap\kern.1mm\kern.#2mm} 
\def\cupss#1#2{\kern.1mm\kern.#1mm\cup\kern.1mm\kern.#2mm} 
\def\cuppp{\kern0.15mm\cup\kern0.15mm} 
\def\timesn{\kern-.2mm\times\kern-.2mm} 
\def\Times{\kern.7mm\hbox{\font\SweD =cmbsy10\SweD \char'002}\kern.7mm}
\def\ttimes{\hbox{\kern-.2mm${}\times\kern-2.5mm\lower.8mm\hbox{\font\SweD =cmr5\SweD t}\kern1.8mm$}} 
\def\ttimesn{\hbox{\kern-.2mm${}\times\kern-2.5mm\lower.8mm\hbox{\font\SweD =cmr5\SweD t}\kern1.4mm$}} 
\def\ktimes{\hbox{\kern-.2mm${}\times\kern-2.5mm\lower1mm\hbox{\font\SweD =cmr5\SweD k}\kern1.5mm$}} 
\def\ltimes{\hbox{${}\times\kern-2.45mm\lower.8mm\hbox{\font\SweD =cmr5\SweD l}\kern1.6mm$}} 
\def\vstimes{\kern.95mm\raise.45mm\hbox{\font\SweD =cmbsy6\SweD \char'002}\kern-2.3mm\lower.9mm\hbox{\font\SweD =cmr5\SweD vs}\kern1.05mm} 
\def\atimes{\kern.8mm\hbox{\font\SweD =cmsy10\SweD \char'002}\kern-1.85mm\lower.65mm\hbox{\font\SweD =cmr5\SweD a}\kern1.4mm} 
\newcommand\meatimes{\kern1mm\raise.5mm\hbox{\font\SweD =cmbsy5\SweD \char'012}\kern1mm} 

\newcommand\circss[2]{\kern.1mm\kern.#1mm\circ\kern.#2mm\kern.1mm} 
\newcommand\mcircss[2]{\kern.#1mm\kern.7mm{\circ}\kern-1.85mm\lower.75mm\hbox{\font\Å=cmr5\Åm}\kern.8mm\kern.#2mm} 
\def\Circ{\kern.9mm\hbox{\font\SweD =cmbsy10\SweD \char'016}\kern.9mm}
\def\cardplus{\hbox{$\kern.77mm+\kern-1.95mm\raise.23mm\hbox{$_{_{\roman c}}$}\kern1.33mm$}}%
\def\ordplus{\hbox{$\kern.78mm+\kern-1.97mm\raise.23mm\hbox{$_{_{\roman o}}$}\kern1.22mm$}}%

\def\ccdot{\hbox{$\kern.77mm\cdot\kern-1mm\raise.45mm\hbox{$_{_{\roman c}}$}\kern.38mm$}} 
\def\svs#1{\sbi{\fiveroman{svs\,}#1}} 
\newcommand\vvs[1]{{_{\kern-0.1mm}}_{\hbox{\font\SweD =cmr5\SweD vs\kern.5mm}#1}} 

\def\Examplee{{\font\SweD =cmssi10\SweD E\kern.15mmx\kern.15mma\kern.15mmm\kern.14mmp\kern.17mml\kern.15mme}\kern.3mm. }
\def\Examples{{\font\SweD =cmssi10\SweD E\kern.15mmx\kern.15mma\kern.15mmm\kern.14mmp\kern.17mml\kern.15mme\kern.15mms}\kern.3mm. }
\def\Remarkk{{\font\SweD =cmssi10\SweD R\kern.15mme\kern.15mmm\kern.15mma\kern.15mmr\kern.15mmk\kern.15mm. }}
\def\Remarkss{{\font\SweD =cmssi10\SweD R\kern.15mme\kern.15mmm\kern.15mma\kern.15mmr\kern.15mmk\kern.15mms}\kern.3mm. }
\def\No{{I\!\!N\kern-.54mm\lower.15mm\hbox{$_{\rm o}$}}} 
\def\iNo{I\!\!{N_{}}_{\kern-.22mm{\rm o}}} 
\def\Nopot#1{I\!\!N\kern-.54mm\lower.15mm\hbox{$_{\rm o}$}\kern-.7mm{}^{#1}} 
\def\potNo{^{\kern.37mm I\!\!{N_{}}_{\kern-.22mm{\rm o}}}} 
\def\minus{\kern.2mm\lower1.05mm\hbox{$^-$}}
\def\pplus{\raise.22mm\hbox{\font\SweD =cmr5\SweD \char'053}}
\def\mminus{\raise.18mm\hbox{\font\SweD =cmsy5\SweD \char'000}}
\def\plusinftyy{{\raise.18mm\hbox{\font\SweD =cmr5\SweD \char'053}\infty}}
\def\minusinftyy{{\raise.18mm\hbox{\font\SweD =cmsy5\SweD \char'000}\infty}}
\def\inftyy{{\raise.15mm\hbox{\font\SweD =cmsy7\SweD \char'061}}} 
\def\inftyyplus{\raise.15mm\hbox{\font\SweD =cmsy7\SweD \char'061}\raise.65mm\hbox{\font\SweD =cmr5\SweD +}} 
\def\plusinfty{\lower1.05mm\hbox{$^+$}\infty}
\def\minusinfty{\lower1.05mm\hbox{$^-$}\infty}

\newcommand\scrmt[1]{\hbox{\font\SweD=eusm10\SweD#1}} 
\newcommand\vcalO{\kern-.1mm\vec{\kern.3mm\hbox{\font\sweD =cmsy7\sweD O}}\kern.2mm} 

\newcommand\conccc{\lower.3mm\hbox{${\bf{\hat{\phantom w}}}$}} 
\newcommand\ssconc[2]{\kern.#1mm
                                \hbox{${\bf{\hat{\phantom w}}}$}\kern.#2mm} 
\def\id{\kern.3mm\roman{id}\kern.7mm}
\def\idv{\hbox{\font\SweD =cmr10\SweD id}\kern.25mm\lower.7mm\hbox{\font\SweD =cmr6\SweD v}\kern.6mm} 
\def\idm{\hbox{\font\SweD =cmr10\SweD id}\kern.25mm\lower.7mm\hbox{\font\SweD =cmr6\SweD m}\kern.6mm} 

\newcommand\seqss[3]{\langle\kern.7mm\kern.#1mm#3\kern.#2mm\kern.8mm\rangle} 

\def\vecs{\upsilon\kern-0.3mm\lower.15mm\hbox{$_s$}\kern0.3mm} 
\def\vecss{\hbox{\font\SweD =cmitt10\SweD v}\kern-0.1mm\lower.15mm\hbox{$_s$}\kern0.2mm} 
\newcommand\svecss{\hbox{\font\SweD =cmitt7\SweD v}\kern-0.1mm\lower.45mm\hbox{\font\SweD =cmmi5\SweD s}\kern0.2mm} 
\def\nullv{0\lower.7mm\hbox{\font\SweD =cmr6\SweD v}\kern.6mm} 
\def\nullsv{0\lower.7mm\hbox{\font\SweD =cmr6\SweD sv}\kern.6mm} 
\def\Bnull_#1{\hbox{\font\SweD =cmssbx10\SweD 0}{_{\kern-0.1mm}}_{#1\kern.15mm}} 
\def\Bzero_#1{\hbox{\font\SweD =cmbx10\SweD 0}{_{\kern-0.1mm}}_{#1}} 
\def\bnull#1{\hbox{\font\SweD =cmssbx10\SweD 0}{}_{\font\SweD =cmmi6\lower.15mm\hbox{\kern-.1mm\SweD #1\kern.15mm}}} 
\def\bnulla#1_#2{\hbox{\font\SweD =cmssbx10\SweD 0}{}_{\font\SweD =cmmi6\lower.15mm\hbox{\SweD #1\kern-.1mm}}\lower.3mm\hbox{$_{_{#2}}$}} 
\def\bzero#1{\hbox{\font\SweD =cmbx10\SweD 0}{}_{\font\SweD =cmmi6\lower.15mm\hbox{\kern-.1mm\SweD #1\kern.15mm}}} 
\def\dom{{{}^{}{\rm dom}\,{}_{{}^{}}}}
\def\domm{\kern0.15mm{\rm dom}{^{\kern.3mm\hbox{\font\SweD =cmr6\SweD 2}}}\,}
\def\domr#1{\roman{dom}^{\font\SweD =cmr6\raise.0mm\hbox{\kern.3mm\SweD #1}}}

\def\rng{{}^{}{\rm rng}\,{}_{{}^{}}}
\def\CeeFB^#1{C\kern-0.15mm\lower.85mm\hbox{\font\SweD =cmr5\SweD FB}\kern-2.4mm^{#1}\kern.1mm} 
\def\CeePi^#1{C\kern-0.15mm\lower.85mm\hbox{\font\SweD =cmr5\SweD \char'005}\kern-1mm^{#1}\kern.1mm} 
\def\Ceece^#1{C\kern-0.15mm\lower.8mm\hbox{\font\SweD =cmr7\SweD c}\kern-.5mm^{#1}\kern.1mm} 
\def\CeePii^#1#2(#3){C\kern-0.15mm\lower.85mm\hbox{\font\SweD =cmr5\SweD \char'005}\kern-1mm^{#1}\kern.1mm(\kern.#2pt\boldsymbol#3\kern0.37mm)} 
\def\Ceecee^#1#2(#3){C\kern-0.15mm\lower.8mm\hbox{\font\SweD =cmr7\SweD c}\kern-.5mm^{#1}\kern.1mm(\kern.#2pt\boldsymbol#3\kern0.37mm)} 
\newcommand\Cinftycee{C\kern-.1mm\lower.8mm\hbox{\font\Å=cmr7\Åc}\kern-.32mm\raise.3mm\hbox{$^\infty$}\kern.3mm} 
\def\Ccinfty{C\lower.3mm\hbox{$\kern-0.2mm_{\roman c}$}\kern-.6mm\raise.3mm\hbox{$^\infty$}\kern0.15mm} 
\def\CPi#1{C\kern-.2mm\lower.05mm\hbox{$_{_\Pi}$}\kern-1.52mm{}^{#1}}
\def\CinftyPi{C\kern.65mm\raise.3mm\hbox{$^\infty$}\kern-3.6mm_{_\Pi}\kern1.55mm} 
\def\CinftyS{\Cinfty\kern-3.9mm_{_{\Cal S}}\kern1.45mm}

\def\RHB#1#2{\raise#1mm\hbox{$#2$}} 
\def\LHB#1#2{\lower#1mm\hbox{$#2$}} 

\def\fiveroman#1{\hbox{\font\SweD =cmr5\SweD #1\kern.1mm}}

\def\sixroman#1{\hbox{\font\SweD =cmr6\SweD #1\kern.1mm}}
\def\eightmath#1{\hbox{\font\SweD =cmmi8\SweD {#1}\kern.1mm}}
\newcommand\emath[1]{\hbox{\font\SweD =cmmi8\SweD {#1}\kern.1mm}}
\newcommand\emth[2]{\lower.0#2mm\hbox{\raise.#2pt\hbox{\font\SweD =cmmi8\SweD #1}}} 
\def\eightroman#1{\hbox{\font\SweD =cmr8\SweD {#1}\kern.1mm}}
\def\subtext#1{\raise.2mm\hbox{$_{_{\kern0.15mm\roman{#1}}}$}}
\newcommand\subText[1]{_{\lower.15mm\hbox{\font\SweD =cmr5\SweD #1}}} 
\newcommand\yxbtext[3]{_{\kern.#2mm\lower.#1mm\hbox{\lower.15mm\hbox{\font\SweD =cmr5\SweD #3}}}} 
\def\subtexT#1{\raise.2mm\hbox{$_{_{\kern0.15mm\hbox{\font\SweD =cmr5\SweD #1}}}$}}
\def\ssNor#1_#2{\kern.4mm\kern.#1pt\lower.32mm\hbox{$_{#2}$}} 
\def\sNor#1{\kern.25mm\lower.38mm\hbox{$_{#1}$}}
\def\sNorr#1{\kern-.2mm\lower.38mm\hbox{$_{#1}$}}
\def\sNoreset_#1{\kern.13mm\lower.83mm\hbox{\font\SweD =cmmi6\SweD C}\kern.32mm\lower.1mm\hbox{$_{^{\emptyset,#1}}$}}
\def\norSig_#1{|\kern.23mm\lower.87mm\hbox{\font\å=cmr5\å\char'006\kern.15mm{#1}}} 
\def\sbi#1{{_{\kern-0.1mm}}_{#1}} 
\def\ais#1_#2{{}_{\font\SweD =cmmi6\lower.15mm\hbox{\kern-.1mm\SweD #1\kern.15mm}}\lower.3mm\hbox{${_{\kern-0.3mm_{#2}}}$}} %
\def\aais#1_#2{\kern.1mm{}_{\font\SweD =cmmi6\lower.25mm\hbox{\kern-.1mm\SweD #1\kern.15mm}}\lower.4mm\hbox{${_{\kern-0.3mm_{#2}}}$}} %
\def\ai#1{{}_{\font\SweD =cmmi6\lower.15mm\hbox{\kern-.1mm\SweD #1\kern.15mm}}} 
\def\yi#1{^{\font\SweD =cmmi6\raise.0mm\hbox{\kern-.1mm\SweD #1\kern.15mm}}} 
\def\ear#1{{}_{\font\SweD =cmr5\lower.15mm\hbox{\kern.1mm\SweD #1}}} 
\def\ar#1{{}_{\font\SweD =cmr6\lower.15mm\hbox{\kern.1mm\SweD #1}}} 
\newcommand\aR[1]{_{\lower.15mm\hbox{\font\SweD =cmr6\SweD #1}}} 
\def\aar#1{_{\font\SweD =cmr6\lower.15mm\hbox{\kern.1mm\SweD #1}}} 
\def\aars#1_#2{{#1_{\kern-.1mm}}_{#2}} 
\newcommand\Yr[2]{\raise1.5mm\hbox{\font\å=cmr5\å#1\kern.#2mm}}%
\def\yr#1{^{\font\SweD =cmr6\raise.0mm\hbox{\kern.3mm\SweD #1}}} 
\def\yrai^#1_#2{^{\kern.4mm\hbox{\font\SweD =cmr6\SweD {#1}}}_{\kern.2mm{#2}}}
\def\upparentes#1{^{\kern.2mm\raise.2mm\hbox{\font\SweD =cmr6\SweD \char'050}\kern.1mm{#1}\kern.1mm\raise.2mm\hbox{\font\SweD =cmr6\SweD \char'051}}} 
\def\lupar{\kern.2mm\lower1mm\hbox{$^{^(}$}} 
\def\rupar{\lower1mm\hbox{$^{^)}$}\kern-.15mm} 
\def\yyi#1{^{\font\SweD =cmmi6\lower.6mm\hbox{\kern-.25mm\SweD #1\kern-.05mm}}} 
\def\yyr#1{^{\font\SweD =cmr6\lower.45mm\hbox{\kern-.25mm\SweD #1\kern-.15mm}}} 
\def\yplus{\lower1mm\hbox{$^{^+}$}} 
\def\yminus{\lower1mm\hbox{$^{^-}$}} 
\def\aminus{{\kern.15mm\raise.3mm\hbox{$_{_-}$}\kern-.1mm}}%
\def\yvee{\LHB{.9}{^{^{\,\vee}}}\kern-.3mm} 
\def\ywed{\LHB{.9}{^{^{\,\wedge}}}\kern-.3mm} 
\newcommand\myprime{\kern.3mm\raise1.5mm\hbox{\font\SweD =cmsy5\SweD \char'060}} 
\newcommand\myprimes{\kern.3mm\raise1.5mm\hbox{\font\SweD =cmsy5\SweD \char'060\kern-.3mm\char'060}} 
\def\adot{\kern.2mm\hbox{\font\SweD =cmb10\SweD \char'056}}%
\def\aadot{\kern.1mm\lower.1mm\hbox{\font\SweD =cmb7\SweD \char'056}}%
\def\ydot{\kern.2mm\raise1.9mm\hbox{\font\SweD =cmb10\SweD \char'056}}
\def\yydot{\kern.2mm\raise1.35mm\hbox{\font\SweD =cmb7\SweD \char'056}\kern.2mm}
\def\yydott{\kern.2mm\raise1.35mm\hbox{\font\SweD =cmb6\SweD \char'056}\kern.2mm}
\def\ClTopR#1^#2{\roman{Cl}\kern.6mm_{\hbox{\font\å=eusm6\åT}_{\mathbb R}^{#2}}\kern.#1mm} 
\def\ClT{{\rm Cl}\kern.25mm\lower.4mm\hbox{$_{\Cal T}$}\kern0.2mm} 
\def\IntT{\sp{\rm Int}\kern.2mm\lower.4mm\hbox{$_{\Cal T}$}\kern0.2mm} 
\def\Cl_taurd#1{\roman{Cl_{}}_{\kern0.37mm\hbox{\font\SweD =cmmi8\SweD \char'034}\kern-0.15mm{_{}}_{\roman{rd}}\kern0.2mm#1\,}}
\def\Int_taurd#1{\roman{Int_{}}_{\kern0.37mm\hbox{\font\SweD =cmmi8\SweD \char'034}\kern-0.15mm{_{}}_{\roman{rd}}\kern0.2mm#1\,}}
\def\inc{\subseteq}

\def\exi#1{\exists\,#1\kern.2mm\,;}
\def\all#1{\forall\,#1\kern.2mm\,;}
\def\imply{\Rightarrow}

\def\sP#1 {\kern.#1mm}
\newcommand\sppp{_{\kern0.2mm}} 
\def\spp{\kern0.07mm} 
\def\sp{\kern0.15mm} 
\def\ssp{\kern0.37mm} 
\def\snn{\kern-0.2mm} 
\def\sn{\kern-0.3mm} 
\def\ssn{\kern-0.63mm} 
\def\biggerlineskip#1 {\linebreak\nopagebreak\vskip-4.2mm\vskip.#1mm\nopagebreak\noindent}%
\def\Biggerlineskip#1 {\linebreak\nopagebreak\vskip-4.2mm\vskip#1pt\nopagebreak\noindent}%
\def\nKP#1{$\null$\kern#1mm}
\def\nKN#1\par{$\null$\kern-#1mm} 
\newcommand\KPt[1]{\kern.#1mm} 
\def\KP#1{\kern#1mm} 
\def\KN#1{\kern-#1mm} 
\def\nhskip#1mm{$\null$\kern#1mm}

\def\mhyppy#1{\null\kern#1mm}

\def\text#1{\hbox{\rm#1}}

\def\VBOX/#1/#2/HEREend{\vbox{#2\vskip-#1mm}\vfill\null\eject}
\def\PouN$#1${\hbox{$#1$}} 
\def\?$#1${\hbox{$#1$}} 
\def\"{\"a} \def\"{\"o}
\def\q#1{``\kern0.37mm#1\kern0.37mm"}
\def\newProCla#1\par#2\par{\vskip1.7mm\noindent\bf#1\it#2\vskip1.7mm}
\def\Prooff{{\font\SweD =cmssi10\SweD P\kern.37mmr\kern.37mmo\kern.37mmo\kern.37mmf\kern.37mm. }\rm}

\def\QED{\hfill\hbox{$\ \sqcap\kern-2.45mm\sqcup$}}

\def\noin{\noindent}
\def\Newline{\kern-10mm\newline}

\def\eps{\varepsilon}

\def\inskipline#1#2 \par{\vskip#1mm$\null$\kern-4.25mm\kern#2mm}

\def\sigrd{\sigma\kern-.2mm\lower.7mm\hbox{\font\SweD =cmr6\SweD r\font\SweD =cmr5\SweD d}\kern.6mm}
\def\ssigrd{\sigma\kern-.2mm\lower.7mm\hbox{\font\SweD =cmr6\SweD r\font\SweD =cmr5\SweD d}\kern-1.7mm\raise1.25mm\hbox{\font\SweD =cmr6\SweD 2}\kern1mm}
\def\sssigrd{\sigma\kern-.2mm\lower.7mm\hbox{\font\SweD =cmr6\SweD r\font\SweD =cmr5\SweD d}\kern-1.7mm\raise1.25mm\hbox{\font\SweD =cmr6\SweD 3}\kern1mm}
\def\sigrdu^#1{\sigma\kern-.2mm\lower.7mm\hbox{\font\SweD =cmr6\SweD r\font\SweD =cmr5\SweD d}\kern-1.7mm\raise1.25mm\hbox{\font\SweD =cmr6\SweD #1}\kern1mm}
\def\taurd{\tau\kern-.4mm\lower.7mm\hbox{\font\SweD =cmr6\SweD r\font\SweD =cmr5\SweD d}\kern.6mm}
\def\ttaurd{\tau\kern-.4mm\lower.7mm\hbox{\font\SweD =cmr6\SweD r\font\SweD =cmr5\SweD d}\kern-1.7mm\raise1.25mm\hbox{\font\SweD =cmr6\SweD 2}\kern1mm}
\def\tsigrd{\tau\sigma\kern-.2mm\lower.7mm\hbox{\font\SweD =cmr6\SweD r\font\SweD =cmr5\SweD d}\kern.6mm}
\def\tauRe{\tau{_{\kern-0.6mm}}_{\hbox{\font\SweD =cmmi5\SweD I\!\!R}}} 
\def\bartauRe{\bar\tau{_{\kern-0.6mm}}_{\hbox{\font\SweD =cmmi5\SweD I\!\!R}}} 
\def\tauR#1{\tau_{_{I\!\!R}}\kern-1.5mm^{#1}}
\def\RN{I\!\!R\kern.3mm^{\hbox{\font\SweD =cmmi6\SweD N}}} 
\def\QTN{Q\kern.1mm_{\lower.2mm\hbox{\font\SweD =cmmi6\SweD T}}^{\kern.2mm\hbox{\font\SweD =cmmi6\SweD N}}} 
\def\lemod_#1(#2){\kern.2mm\raise1.3mm\hbox{\font\SweD =cmsy5\SweD \char'024}\kern-1.65mm\lower.8mm\hbox{\font\SweD =cmr5\SweD #1}\kern.4mm(\boldsymbol{#2}\kern.37mm)} 
\def\preLCS(#1){\kern.2mm\raise1.3mm\hbox{\font\SweD =cmsy5\SweD \char'024}\kern-1.65mm\lower.8mm\hbox{\font\SweD =cmr5\SweD LCS}\kern.4mm(\boldsymbol{#1}\kern.37mm)} 
\def\leLCSv(#1)-{\kern.2mm\raise1.3mm\hbox{\font\SweD =cmsy5\SweD \char'024}\kern-1.65mm\lower.8mm\hbox{\font\SweD =cmr5\SweD LCS}\kern.4mm(\boldsymbol{#1}\kern.37mm)\text{\,-\kern0.15mm}} 
\def\lleLCS(#1)-{\hbox{\font\SweD =cmsy10\SweD \char'024}\kern.3mm\lower.62mm\hbox{\font\SweD =cmr5\SweD LCS}\kern.4mm(\kern0.15mm#1\kern0.37mm)\text{\,-\kern0.15mm}} 
\def\leLCSr{\hbox{\font\SweD =cmsy10\SweD \char'024}\kern.3mm\lower.62mm\hbox{\font\SweD =cmr5\SweD LCS}\kern.4mm}
\def\leLCS-{{\le}{}_{_{{\rm LCS}}}\text{\sp-\sp}}
\def\vspreceq{\kern1.4mm\raise1.7mm\hbox{\font\SweD =cmr5\SweD v}\kern-1.1mm\lower.12mm\hbox{\font\SweD =cmr5\SweD s}\kern-1.5mm\hbox{\font\SweD =cmsy10\SweD \char'026}\kern1mm} 
\def\tvpreceq{\kern1.4mm\raise1.7mm\hbox{\font\SweD =cmr5\SweD v}\kern-1.2mm\lower.25mm\hbox{\font\SweD =cmr5\SweD t}\kern-1.4mm\hbox{\font\SweD =cmsy10\SweD \char'026}\kern1mm} 
\def\Centerline#1\par#2\par#3{\noindent#1\phantom{#3}\hfill#2\hfill\phantom{#1}#3}

\begin{document}

\title[$\text{\sc Real analyticity is shy}$]%
             {Real analyticity of composition is shy}

\author[S. Hiltunen]{Seppo\ I\. Hiltunen}
\address{Aalto University                                               \vskip0mm$\hspace{2mm}$
           Department of Mathematics and Systems Analysis               \vskip0mm$\hspace{2mm}$
           P.O.\ Box 11100                                              \vskip0mm$\hspace{2mm}$
           FI-00076 Aalto                                               \vskip0mm
         Finland}
\email{seppo.i.hiltunen\,@\,aalto.fi}



\begin{abstract} 

Dahmen and Schmeding have obtained the result that although the smooth Lie 
group \math{G} of real analytic diffeomorphisms \math{
 \mathbb S^{\,1.}\to\mathbb S^{\,1.}} has a compatible analytic manifold 
structure, it does not make \math{G} a real analytic Lie group since the group 
multiplication is not real analytic. The authors considered this result 
\q{surprising} for the applied concept of infinite\ssp-\sp dimensional real 
analyticity for maps \PouN$\ssp E\to F\spp$, defined by the property that 
locally a holomorphic extension \math{ E\RHB{.3}{\sp_{_{\bbC}}} \to 
 F\RHB{.3}{\sn_{_{\bbC}}} } exist. In this note we show that this type of 
real analyticity is quite rare for composition maps \math{ \roman f\,\varphi : 
 x\mapsto\varphi\circ x} when \math{\varphi} is real analytic. Specifically, 
we show that the smooth Fr\'echet space map \math{ \roman f\,\varphi :
 C\,(\ssbb43 R)\to C\,(\ssbb43 R)} for real \linebreak 
analytic \math{\varphi:\bbR\to\bbR} is real analytic in the above sense only 
if \math{\varphi} is the restriction to \math{\bbR} of \linebreak 
some entire function \mathss36{\bbC\to\bbC}. We also discuss the 
possibility of proving that the set of these \q{admissible} functions \math{
\varphi} be \q{small} in the space \math{A\,(\ssbb43 R)} of real analytic 
functions either in the Baire categorical sense, or in the measure theoretic 
  sense of shyness.
  \end{abstract}

\maketitle


\noin For maps \math{f:E\supseteq U\to F}, understood as triplets \math{
\tilde f=(\sp E\ssp,\sp F\spp,\sp f\ssp)} with \math{U=\dom f}, where \math{E} 
and \math{F} are real Hausdorff locally convex spaces and \math{f} is a 
function between the underlying sets, there are several possibilities to 
reasonably define real analyticity of \math{\tilde f}. One is that of the 
\q{convenient calculus} developed in \cite[p.\ 97\,ff.]{KM}\,. Another 
possibility is to represent \math{f} locally in some sense as a limit of 
partial sums of \q{power series}. A third possibility is to require locally 
existence of some \q{holomorphic} extension \math{E\ar{cx}\to F\ar{cx}} 
between the complexifications, cf.\ \cite[pp.\ 51\,--\,52]{Glö}\,. This third 
approach further divides into several possibilities according to what kind of 
concept of holomorphy one chooses to use, cf.\ \cite{BoSi} and \cite{Piz}\,.

In this note, we shall use that third approach with holomorphy defined as 
meaning being $\Cinfty$ between complex Hausdorff locally convex spaces in the 
sense of \cite{HiSM} with topological vector spaces being interpreted as 
convergence vector spaces as explained there on page 236\ssp. We let \math{
\CeePi^\inftyy(\sp\tfbbC\ssp)} denote the class of thus obtained holomorphic 
maps \math{\tilde f}. As explained in \cite[Remarks 0.12\ssp, p.\ 241]{HiSM} 
for real scalars, noting that taking \math{\bbC} in place of \math{\bbR} in 
the required proofs does not change anything essential, our concept of 
holomorphy is precisely the same as that in \cite[p.\ 23]{DS}\,. Hence also 
our associated real analyticity is precisely the same as there. See further 
\cite[Theorem 3.8\ssp, pp.\ 14\ssp, 18]{HiDim} for the case where \math{F} is 
  Mackey complete.

Let us say that \math{\tilde f} is {\it conveniently real analytic\ssp} in the 
case where \PouN$\ssp f:\dom f\to F$ is real analytic in the sense of 
\cite[Definition 10.3\ssp, p.\ 102]{KM}\,.

For \math{E} in the class \math{\LCSps5(\tfbbR)} of all real Hausdorff locally 
convex spaces, we let \math{E\ar{cx}} denote the {\it complexification\ssp} as 
explained in \cite[A\,2\ssp, p.\ 23]{DS}\,. Hence then \PouN$ E\ar{cx} \in 
 \LCSps5(\tfbbC) \ssp$ holds, and for the underlying sets we have \mathss37{
\vecs(\sp E\ar{cx}\sp)=\vecs E\sp\times\vecs E}. Here \math{\tfbbR} and \math{
\tfbbC} are the real and complex topological fields whose underlying sets are \math{
\bbR} and \mathss38{\bbC}, \,respectively.

We will consider the maps \mathss37{ \roman F\,\varphi = 
 (\sp E\ssp,\sp E\ssp,\sp\roman f\,\varphi\ssp) }, \,where \math{ E = 
 C\,(\ssbb43 R)} is the real Fr\'echet space of continuous functions \math{
\bbR\to\bbR} with topology that of uniform con- \linebreak 
                                                vergence on bounded intervals, 
and \math{\roman f\,\varphi=\seqss33{\varphi\circ x:x\in\vecs E} } with \math{
\varphi:\bbR\to\bbR} real \linebreak 
                          analytic. Thus \math{\roman f\,\varphi} is the 
function \math{\vecs E\to\vecs E} defined by \mathss35{ 
  x \mapsto \varphi\circ x }.

From our Theorem \ref{Tsmoo} below it follows that \math{\roman F\,\varphi} is 
smooth in all reasonable senses, and also conveniently real analytic. Contrary 
to this, by Theorem \ref{mainTh} it is real analytic in our sense only if \math{
\varphi} has an entire extension.

Below, we let \math{f\sp\fvalue x} be the function value of \math{f} at \math{
x} instead of the usual \q{$f\ssp(\sp x\sp)$}. The zero vector of a 
topological (\sp or any structured\,) vector space \math{F} is \mathss38{
\Bnull_F}. In \linebreak 
              particular, for our fixed \math{E} above we have \mathss38{
\Bnull_E=\bbR\times\snn\{\ssp 0\ssp\} }. We let \math{\Univ} be the class of 
all sets, and for functions \math{f} and \math{g} we have \math{
[\KP1 f\sp,\sp g\rbrakf} the function defined on \linebreak \PouN$
                                                  \dom f\capss21\dom g\ssp$ by \mathss37{
x\mapsto(\ssp f\sp\fvalue x\ssp,\sp g\fvalue x\ssp) }. We put \mathss38{
\roman{pr}\ar 1 = \{\,(\ssp x\ssp,\sp y\ssp,\sp x\ssp) : 
 x\ssp,\sp y\in\Univ\KP1\} }, \,the global \q{first projection}, and \mathss38{
(\ssp z\,;\sp x\ssp,\sp y\ssp) = (\ssp z\ssp,\sp(\ssp x\ssp,\sp y\ssp)) }. If \math{
z = (\ssp x\ssp,\sp y\ssp) } is an ordered pair, then \math{x=\sigrd z} and \mathss36{
y=\taurd z }. We further refer to \cite[pp.\ 4\,--\,8]{HiDim}\,, 
\cite[pp.\ 4\,--\,9]{SeBGN} and \cite[p.\ 1]{FKBGN} for a more extensive 
  explanation of our notational system.

\begin{theorem}\label{Tsmoo}

If \PouN$\,\varphi:\bbR\to\bbR$ is smooth{\ssp\rm, }then \PouN$\,
\roman F\,\varphi\in\CinftyPi(\sp\tfbbR\ssp)$ holds. If \PouN$\,
\varphi:\bbR\to\bbR$ is real analytic{\ssp\rm, }then $\,\roman F\,\varphi$ is 
  conveniently real analytic. 
  \end{theorem}

\begin{proof} The first assertion follows from 
\cite[Theorem 3.6\ssp, p.\ 17]{HiDim} similarly as (a) in Remarks 3.7 there. 
For the second assertion, assuming the premise, by 
\cite[Theorem 10.4\ssp, p.\ 102]{KM} for any continuous linear functional \math{
\ell:E\to\tfbbR} and for arbitrarily fixed \math{x\ssp,\sp u\in\vecs E} it 
suffices to show that the real function \PouN$\ssp\chi:\bbR\to\bbR$ \linebreak 
given by \math{ t \mapsto \ell\circss01\roman f\,\varphi\fvalue(\ssp 
 x + t\,u\ssp) } is real analytic on some open interval around zero. This in 
turn follows if we show that it has a holomorphic extension \math{\bar\chi} 
around zero in some open set of the complex plane.

Now, by the Riesz representation \cite[Theorem 7.4.1\ssp, p.\ 186]{Du} in 
conjunction with 
\cite[Corollary\ssp, p.\ 262\ssp, Proposition 3.14.1\ssp, p.\ 266]{Ho}\,, 
there are a compact interval \linebreak 
                        \PouN$I\inc\bbR\ssp$ and a bounded regular signed 
Borel measure \math{\mu} on \math{I} with the property that \PouN$
\ell\fvalss40 y = \int_{\KP1 I\,}y\rmdss11\mu\ssp$ holds for all \mathss35{
y\in\vecs E}. Since \math{\varphi} has a holomorphic extension $\ssp
\bar\varphi$ \linebreak
             defined on some open set in the complex plane containing \mathss38{
\bbR}, \,and since \math{x\ssp\image\sn I} and \linebreak 
                                               $u\ssp\image\sn I\ssp$ are 
compact, there is \math{\eps\in\rbb R^+} such that for \vskip.5mm\centerline{$
\Omega =\{\,t + \roman i\KPt8\sigma : \minus 1 < t < 1\ssp\text{ and }\ssp
                                    \minus\eps < \sigma < \eps\,\} $} \inskipline{0.5}0

we have \math{x\fvalue s + \zeta\KP1(\ssp u\fvalue s\ssp) \in \dom\bar\varphi} 
for all \math{s\in I} and \mathss35{\zeta\in\Omega}. Then defining \PouN$\ssp
\bar\chi:\Omega\to\bbC$ by \mathss38{ \zeta \mapsto
 \int_{\KP1 I\,}\bar\varphi\fvalss10(\ssp 
 x\fvalue s + \zeta\KP1(\ssp u\fvalue s\ssp)) \rmdss11\mu\,(\sp s\sp) }, \,we 
have \math{\bar\chi} continuous with \mathss36{\bar\chi\KP1|\KP1\bbR\inc\chi}, \,%
and hence we are done if we show that \math{\bar\chi} is holomorphic.

Letting \math{\Gamma} be the positively oriented boundary of an arbitrarily 
fixed closed triangle included in \mathss36{\Omega}, \,by Morera's theorem it 
suffices to show that \PouN$\ssp \oint_{\KP1\Gamma\,}\bar\chi = 0$ \linebreak 
holds. Now applying Fubini's theorem separately to the positive and negative 
part of \math{\mu} in its Jordan decomposition, we obtain \inskipline1{25.5}

$ \oint_{\KP1\Gamma\,}\bar\chi 
 = \oint_{\KP1\Gamma\sp}\int_{\KP1 I\,}
    \bar\varphi\fvalss10(\ssp x\fvalue s + \zeta\KP1(\ssp u\fvalue s\ssp))
     \rmdss11\mu\,(\sp s\sp)\rmdss21\zeta 
$ \inskipline{.5}{33.3}

${} = \int_{\KP1 I\sp}\oint_{\KP1\Gamma\,}
       \bar\varphi\fvalss10(\ssp x\fvalue s + \zeta\KP1(\ssp u\fvalue s\ssp))
        \rmdss11\zeta\rmdss21\mu\,(\sp s\sp) = 0 \KP1$.
  \end{proof}

\begin{theorem}\label{mainTh}

If \PouN$\,\varphi:\bbR\to\bbR$ is real analytic with $\,\roman F\,\varphi$ 
real analytic{\ssp\rm, }then there is a holomorphic $\,\chi:\bbC\to\bbC$ with $\,
  \varphi=\chi\KP1|\KP1\bbR \KPt8$.
  \end{theorem}

\begin{proof} Assuming the premise, with \math{E=C\,(\ssbb43 R)} as above, let \math{
G=E\ar{cx}} and \PouN$\ssp y = $ \PouN$
                               \bbR\times\snn\{\,\varphi\fvalss01 0\,\} \,$. 
Now having \mathss36{(\ssp\Bnull_E\ssp,\sp y\ssp)\in\roman f\,\varphi }, \,%
some \math{g} exists such that \PouN$\ssp 
(\ssp\Bnull_G\ssp;\ssp y\ssp,\sp\Bnull_E\spp) \in g $ and \math{
(\ssp G\sp,\sp G\sp,\sp g\ssp) \in \CeePi^\inftyy(\sp\tfbbC\ssp) } and \math{
 g\KP1|\KP1(\ssp\Univ\times\snn\{\,\Bnull_E\}\spp) \inc 
 [\KP{1.2}\roman f\,\varphi\circss01\roman{pr}\ar 1 \ssp , \sp
   \Univ\times\snn\{\,\Bnull_E\}\rbrakf } hold. \newline
                                                Having \math{ \Bnull_G \in 
\dom g\in\taurd G}, \,there is \math{n\ar 0\in\rbb Z^+} such that for \vskip.5mm\centerline{$
W=\vecs G\capss21\{\,z:\all s\,\minus n\ar 0\le s \le
n\ar 0\imply|\KP1\sigrd z\fvalue s\KP1| + |\KP1\taurd z\fvalue s\KP1| <
n\ar 0\KN1^{\mminus 1}\ssp\big\} $} \inskipline{.5}0 

we have \mathss36{W\inc\dom g}. With 
\math{m\ar 0 = n\ar 0+1} and \math{
 J = {\ssp]\sp}\,n\ar 0\ssp,\plusinfty\,{\sp[\ } } taking \vskip.5mm\centerline{$
 v = (\ssbb30 R\setminusn J\ssp)\times\snn\{\ssp 0\ssp\} \cupss00
       \seqss33{s - n\ar 0\sn : s\in J} \,$ and $\, w = 
     (\sp v\ssp,\Bnull_E) \,$ and} \inskipline{.5}0

\?$\ell = \seqss33{\sigrd z\fvalue m\ar 0 + \sp\roman i\sp\,(\sp
       \taurd z\fvalue m\ar 0):z\in\vecs G} \,$, we obtain \math{
\{\,\zeta\,w:\zeta\in\ssbb02 C\,\}\inc W} with $\ell$ a continuous linear map 
\mathss38{
G\to\tfbbC}. Hence for \math{\chi =
\seqss43{\ell\circss 21 g\fvalue(\ssp\zeta\,w\ssp):\zeta\in\bbC} } we 
have \math{
\chi:\bbC\to\ssbb03 C} holomorphic. In addition, 
for \math{t\in\bbR} we obtain \vskip.5mm \centerline{$
 \chi\fvalss01 t = \ell\circss10 g\fvalue(\ssp t\,v\ssp,\sp\Bnull_E\spp)
  = \ell\circss11[\KP{1.2}\roman f\,\varphi\circss11\roman{pr}\ar 1\ssp,\sp
   \Univ\times\snn\{\,\Bnull_E\}\rbrakf\KN{.5}\fvalue
 (\ssp t\,v\ssp,\sp\Bnull_E\spp) $} \inskipline{.5}{22.7}

${}
 = \ell\fvalss20(\,\roman f\,\varphi\fvalue(\ssp t\,v\ssp)\ssp,\sp\Bnull_E\spp)
 = \roman f\,\varphi\fvalue(\ssp t\,v\ssp)\fvalue m\ar 0
 = \varphi\circ(\ssp t\,v\ssp)\fvalue m\ar 0$ \inskipline{.5}{22.7}

${} 
 = \varphi\fvalue(\,t\KP1(\ssp v\fvalue m\ar 0\sp)) 
 = \varphi\fvalue(\ssp t\,1\ssp)) 
 = \varphi\fvalue t \,$.
  \end{proof}

The argument of the above proof of Theorem \ref{mainTh} {\it does not\ssp} 
work if instead we take the Fr\'echet space \mathss38{ E = 
 \Cinfty\big(\sp[\KP1 0\,,\sp 1\KPt9]\sp\big) }. However, it is obvious that 
the same idea can be used to prove similar results for spaces \math{ E = 
 \Cinfty(\sp\Omega\sp)} when \math{\Omega} is a nonempty open set in some 
\q{nonzero} Euclidean space.

\begin{remark}\label{Remrare}

Letting \math{\Omega} be the set of all real analytic \mathss38{x:\bbR\to\bbR
}, \,and \math{S} its subset formed by the \math{x} possessing an entire 
extension, if one wants to consider whether \linebreak
                                            $S\ssp$ be \q{small} in \math{
\Omega} in some precise sense, one must put some structure on \mathss36{\Omega
}. A standard procedure is to construct the locally convex space \PouN$\ssp 
A\,(\ssbb43 R)=(\sp X\sp,\sp\scrmt T\KP1)=F$ \linebreak
                                             with \math{\vecs F=\Omega} as 
follows. Let \math{X} be the (\sp abstract\ssp) real vector space with 
underlying set \math{\Omega} obtained by taking the \q{obvious} pointwise 
operations, and let \math{\scrmt T} be the strongest locally convex topology 
for \math{X} such that the identity is a continuous linear map \math{
\roman F\,U\to F} for all \mathss37{U\in\scrmt U}, \,when \math{\scrmt U} is 
the set of all open \math{U} in \math{\bbC} with \linebreak
                                            \PouN$\bbR\inc U\spp$, \,and \math{
\roman F\,U} is the \q{obvious} Fr\'echet space of functions \math{x\in\Omega} 
possessing a holomorphic extension \mathss38{U\to\bbC}.

Letting \math{\scrmt B} be the Borel $^\sigma\sn$algebra of the topological 
space \mathss38{(\ssp\Omega\,,\sp\scrmt T\KP1)}, \,now to the above smallness 
one can give a precise mathematical content in one of the following two 
different ways: \inskipline{.5}2

(1) \ in the topological Baire categorical sense, here meaning that \math{S} 
    is \mathss37{\scrmt T}{--\,\it meager\ssp} \linebreak 
in the sense that it can expressed as a countable union of \q{rare} sets, i.e.\ 
those with closure having no interior points. \inskipline{.2}2

(2) \ in the measure theoretic sense, here meaning that \math{S} is \mathss37{
    F}{--\,\it shy\ssp} in the sense that it is contained in a countable union 
of sets \math{B\in\scrmt B} having the property that there is a 
finite\ssp-\sp dimensional subspace \math{M} in \math{X} such that \math{
B\capss21\{\,x+v:v\in M\KP1\} } for \linebreak 
all \math{x\in B} has Lebesgue measure zero in the \q{obvious} sense. \vskip.5mm

Using the result from \cite{Is} that \math{S\in\scrmt B} holds, we can 
establish (2) quite easily. Namely, with any \math{u\in\vecs F\sp\setminus S} 
taking \math{M=\{\,t\,u:t\in\bbR\KP1\} } then as a singleton in a 
one\ssp-\sp dimensional subspace \math{ S\capss21\{\,x+v:v\in M\KP1\} = 
 \{\ssp x\ssp\} } has measure zero. So by Theorem \ref{mainTh} the set \math{
\vecs A\,(\ssbb43 R)\capss21\{\,\varphi:\roman F\,\varphi\ssp\text{ is real 
 analytic }\} } is \mathss37{A\,(\ssbb43 R)}--\,shy.

If instead of \math{A\,(\ssbb43 R)} we had for example \mathss38{ F = 
 A\,\big(\sp[\KP1 0\,,\sp 1\KPt9]\sp\big) }, \,which is a Silva space, an 
inductive limit of a sequence of Banach spaces with compact links, then we 
could easily establish (1) by noting that an elementary complex analysis 
argument using Cauchy's formula shows that \math{S} is contained in a 
countable union of \mathss37{\scrmt T}--\,compact, and hence rare sets. Since 
the sets \math{U\in\scrmt U} in our actual situation are unbounded, this 
argument is not applicable. We can only show that \math{S} can be expressed as 
an {\it uncountable\ssp} union of \mathss37{\taurd\roman F\,U}--\,compact, and 
hence \mathss37{\scrmt T}--\,compact sets. So the question whether (1) holds 
in the above situation remains open.
  \end{remark}

\begin{remark}

When defining our concept of a set being \q{shy} in a topological vector 
space, we above deviated e.g.\ from the approaches in \cite{Chris} and 
\cite{Hunt} since we wish to apply the concept to \q{highly nonmetrizable} 
spaces contrary to the cases {\it loc.\ cit.}\ where the underlying topology 
is assumed to be Polish, i.e.\ separable and completely metrizable. 
Specifically, we explicitly required the set to be contained in a countable 
union of \q{negligible} Borel sets since otherwise it might happen that a 
countable union of shy sets is not shy. In the restricted case of Polish 
topologies, one is able to give a quite nontrivial proof that a countable 
union of negligible Borel sets also is such. See 
  \cite[pp.\ 223\,--\,224]{Hunt} for the details.

Note further that for example for \mathss36{ F = 
 \tfbbR\KPt8^{\sbbNo\sp)}\RHB{.62}{\sp\fiveroman{lcx}} }, \,the countable 
direct sum of the topological field \mathss38{\tfbbR}, \,we trivially have (1) 
that \math{\vecs F} is \mathss37{\taurd F}--\,meager, and (2) that \math{
\vecs F} is \mathss37{F}--\,shy. So an infinite\ssp-\sp dimensional locally 
convex space can be both meager and shy \q{in itself} in the sense we defined 
in Remark \ref{Remrare} above. Another example is the Silva space \mathss38{ F 
 = A\,\big(\sp[\KP1 0\,,\sp 1\KPt9]\sp\big) }. In these cases \math{
(\ssp\vecs F\spp,\sp\taurd F\ssp)} is not a Baire topological space, and being 
\q{shy} in \math{F} is not even defined in the sense of \cite{Hunt} since \math{
\taurd F} is not a metrizable topology.
  \end{remark}

If one wished to define shyness and its complement \q{prevalence} more 
carefully and generally, the generated $^\sigma\sn$algebra of \math{\scrmt A} 
being defined by \vskip.5mm\centerline{$
\sigmAlg5\scrmt A = \bigcap\,\{\,\scrmt B:\scrmt B\ssp\text{ is a $^\sigma\sn
 $algebra and }\ssp \scrmt A \inc \scrmt B\KP1\} \KP1$,} \inskipline{.5}0

one could put the following \q{semiformal}

\begin{definitions}\label{dfShy}

(1) \ Say that \math{G} is a {\it topologized group\ssp} if{}f there are \math{
    g\,,\sp\Omega\,,\sp\scrmt T} such that \PouN$G=(\ssp g\,,\sp\scrmt T\KP1) \ssp$ 
and \math{(\ssp\Omega\,,\sp\scrmt T\KP1)} is a topological Hausdorff space and \math{
g} is a group operation on \mathss36{\Omega}, \,and for all \math{x\in\Omega} 
  it holds that \vskip.2mm\centerline{$
\seqss33{g\fvalue(\ssp x\ssp,\sp y\ssp):y\in\Omega} \ssp$ and \math{
\seqss33{g\fvalue(\ssp y\ssp,\sp x\ssp):y\in\Omega} }     are continuous \mathss34{
 \scrmt T\to\scrmt T }.} \inskipline{.5}2

(2) \ Say that \math{S} is {\it shy\ssp} in \math{G} if{}f \math{G} is a 
    topologized group and for all \math{g\,,\sp\Omega\,,\sp\scrmt T} from \inskipline09

\PouN$G=(\ssp g\,,\sp\scrmt T\KP1)\ssp$ and \math{\Omega=\bigcup\,\scrmt T} it 
follows existence of a countable \PouN$\ssp\scrmt A\inc\sigmAlg2\scrmt T$ \linebreak
with \mathss38{S\inc\bigcup\,\scrmt A }, \,and such that for every \math{A\in
 \scrmt A} there are some \mathss38{\scrmt T}--\,compact $\ssp K$ \linebreak
and a probability measure \math{\mu} with \math{\dom\mu=\sigmAlg2\scrmt T} and \mathss34{
(\sp K\sp,\sp 1\ssp)\in\mu }, \,and such that \linebreak
$\{\KP1 g\fvalue(\ssp g\fvalue(\ssp x\ssp,\sp z\ssp)\ssp,\sp y\ssp) : 
 z\in A\KP1\} \in \mu\invss44\image\snn\{\ssp 0\ssp\} \ssp$ holds for all \mathss37{
  x\ssp,\sp y\in\Omega}. \inskipline{.5}2

(3) \ Say that \math{S} is {\it prevalent\ssp} in \math{G} if{}f for all \math{
    \Omega} from \math{\Omega=\bigcup\,\taurd G} it follows \inskipline0{8.9}

that \math{S\inc\Omega} and \math{\Omega} is not shy in \math{G} and \math{
\Omega\sp\setminus S} is shy in \mathss33{G}. \inskipline{.5}2

(4) \  Say that \math{F} is {\it shy$^{\,{\fiveroman{LCS}}}$ in itself\ssp} 
    if{}f there is \math{\bosy K\in\setRC} with \math{F\in\LCSps0(K)} \inskipline0{8.9}

and such that \math{\vecs F} is shy in \mathss37{
  (\ssp\ssigrd F\spp,\sp\taurd F\ssp) }.
  \end{definitions}

Note above that \q{\math{\Omega} is not shy in \math{G}} is to be implicitly 
understood to mean that \q{it does not hold that \math{\Omega} is shy in \math{
G}}. Further observe that \ref{Remrare}\,(2) is a particular case of 
\ref{dfShy}\,(2) since one can first restrict {\sl a\ssp} Lebesgue measure to 
some \q{cube} of measure one, and then extend it by zero to all Borel sets.

As an application of our Definitions \ref{dfShy} above, we give the following

\begin{proposition}

Every infinite\ssp-\sp dimensional Silva space is shy$^{\,{\fiveroman{LCS}}}$ 
  in itself.
  \end{proposition}

\begin{proof} With \math{\bosy K\in\setRC} letting \math{F\in\LCSps0(K)} be an 
infinite\ssp-\sp dimensional Silva space, there are \math{ \bmii8 F \,\in 
 \BaSps0(K)\KP1^\sbbNo} and \math{\bosy\nu\in\Univ\KP1^\sbbNo} with \PouN$\ssp 
F=\lleLCS(\snn\bosy K)-\inf\ssp\rng\bmii8 F$ and such that for all \math{i\in
 \bbNo} we have \math{\bmii8 F\fvalss81 i\ssp\yplus \tvpreceq
                      \bmii8 F\fvalss81 i} with \math{\bosy\nu\fvalss01 i} a 
compatible norm for \math{\bmii8 F\fvalss81 i} such that \math{
(\ssp\bosy\nu\fvalss01 i\ssp)\invss14\image[\KP1 0\,,\sp 1\KPt9] } is \mathss37{
\taurd\spp(\ssp\bmii8 F\fvalss81 i\ssp\yplus\sp) }--\,compact and \PouN$
(\ssp\bosy\nu\fvalss01 i\ssp)\invss14\image[\KP1 0\,,\sp 1\KPt9] \inc$ \PouN$
(\ssp\bosy\nu\fvalss01 i\ssp\yplus\sp)\invss14\image[\KP1 0\,,\sp 1\KPt9] \ssp$ 
holds. Since \math{F} is infinite\ssp-\sp dimensional, we may also arrange 
mat- ters so that \math{ \vecs(\ssp\bmii8 F\fvalss81 i\ssp) \not=
                         \vecs(\ssp\bmii8 F\fvalss81 i\ssp\yplus\sp) } holds.

Now putting \mathss38{ \roman K\KPt8 i\,j = 
 (\ssp\bosy\nu\fvalss01 i\ssp)\invss14\image[\KP1 0\,,\sp 
  j\ssp\yplus\sp\ydot\KP1] }, \,we have \mathss38{ \vecs F \inc 
 \bigcup\,\{\KP1\roman K\KPt8 i\,j:i\ssp,\sp j\in\bbNo\,\} }, and for 
arbitrarily fixed \math{i\ssp,\sp j\in\bbNo} it remains to construct some \mathss37{
\taurd F}--\,compact \linebreak
                     $K\ssp$ and a probability measure \math{\mu} with \math{
\dom\mu=\sigmAlg2\taurd F} and \mathss37{(\sp K\sp,\sp 1\ssp)\in\mu }, \,and 
such that \math{\{\,(\ssp x + z\ssp)\svs F :  z\in \roman K\KPt8 i\,j\KP1\} 
 \in \mu\invss44\image\snn\{\ssp 0\ssp\} } holds for all \math{x\in\vecs F}. 
For this, we use a classical result of Alexandroff and Urysohn, see 
\cite[Problem O\,(e)\ssp, p.\ 166]{Ky}\,, guaranteeing existence of a 
surjection \math{\chi:\bigcup\,\scrmt T\to\roman K\KPt8 i\,j} which is 
continuous \PouN$\scrmt T\to\taurd(\ssp\bmii8 F\fvalss81 i\ssp\yplus\sp) \,$, \,%
when we take \mathss33{ \scrmt T = 
 \Pows(\ssp 2\sp\adot\sp)\expnota^\ssp\sbbNo]_{ti} }. On \math{
\bigcup\,\scrmt T} we then take the countable product measure \math{\mu\ar 0} 
of \mathss38{ \{\,(\ssp\emptyset\,,\sp 0\ssp)\ssp,
                  (\ssp 1\sp\adot\ssp,\sp\frac 12\ssp\big)\ssp,
              \big(\ssp\{\,1\sp\adot\ssp\}\ssp,\sp\frac 12\ssp\big)\ssp,
                  (\ssp 2\sp\adot\ssp,\sp 1\ssp)\,\} }, \,and fixing any \PouN$
x\ar 0\in\vecs(\ssp\bmii8 F\fvalss81 i\ssp\yplus\sp) \setminus
         \vecs(\ssp\bmii8 F\fvalss81 i\ssp)\ssp$ we put \PouN$\ssp K = 
 \{\,(\ssp x + s\,x\ar 0\sp)\svs F : x\in\roman K\KPt8 i\,j\ssp\text{ and }\ssp 
  0\le s\le 1\KPt8\}$ and \math{ \mu = 
 \seqss33{\roman m\,A:A\in\sigmAlg2\taurd F} } where \vskip.7mm\centerline{$
\roman m\,A = \int_{\KP{1.2}[\,0\ssp,\ssp 1\ssp]}\,\mu\ar 0\!\fvalue(\ssp
 \chi\invss44\image(\ssp\roman K\KPt8 i\,j\capss21\{\, x : 
 (\ssp x + s\,x\ar 0\sp)\svs F\in A\KP1\}\sp))\rmdss11\Lebmef^{}\sp(\sp s\sp) \KP1
 $.} \inskipline{.7}0

It is a standard exercise in measure theory left to the reader to verify that 
these $\ssp K$ \linebreak and \math{\mu} do the job we wished.
  \end{proof}

Note above that \math{E\tvpreceq F} means that \math{\idv F = \id(\ssp
 \vecs F\ssp)} is a continuous linear map $F\to E\,$. In \cite[p.\ 7]{HiDim} 
this was written (\ssp possibly\sp) ambiguously \q{$ E\le F\ssp$}.



\begin{thebibliography}{99} \def\{\font\Å=cmcsc7\Å}

\bibitem{BoSi}\bibname{J. B{\ochnak} and J. S{\iciak}}\ssp: `Analytic 
  functions in topological vector spaces' {\em Studia Math\sp}. {\bf39} (1971) 
  77\,--\,112.
\bibitem{Chris}\bibname{J. P. R. C{\hristensen}}\ssp: `On sets of Haar measure 
  zero in abelian Polish groups' {\em Israel J. Math\sp}. {\bf13} (1972) 
  255\,--\,260.
\bibitem{DS}\bibname{R. D{\ahmen} and A. S{\chmeding}}\ssp: `The Lie group of 
  real analytic diffeomorphisms is not real analytic' {\em preprint}\ssp{\rm, 
  arXiv:1410.8803v2 [math.DG]}.
\bibitem{Du}\bibname{R. M. D{\udley}}\ssp: {\em Real Analysis and
  Probability\sp}, Wadsworth, Pacific Grove 1989.
\bibitem{Glö}\bibname{H. G{\l\"ockner}}\ssp: `Infinite-dimensional Lie groups 
  without completeness restrictions' {\em Geometry and Analysis on Lie 
  Groups\sp}, Banach Center Publications {\bf55}, Warsaw (2002) 43\,--\,59.
\bibitem{HiSM}\bibname{S. H{\iltunen}}\ssp: `Implicit functions from locally 
  convex spaces to Banach spaces' {\em Studia Math\sp}.\ {\bf134} 3 (1999) 
  235\,--\,250.
\bibitem{HiDim}\bibname{S. H{\iltunen}}\ssp: `Differentiation, implicit 
  functions, and applications to generalized wellposedness' {\em preprint\sp} 
  arXiv:math/0504268v3 [math.FA].
\bibitem{SeBGN}\bysame\ssp: `Seip's differentiability concepts as a 
particular case of the Bertram\,--\,Gl\"ockner\,--\,Neeb construction'
  {\em preprint\sp} arXiv:0708.1556v7 [math.FA].
\bibitem{FKBGN}\bysame\ssp: `The Fr\"olicher\,--\,Kriegl differentiabilities 
  as a particular case of the Bertram\,--\,Gl\"ockner\,--\,Neeb construction' 
  {\em preprint\sp} arXiv:0804.4273v1 [math.FA].
\bibitem{Ho}\bibname{J. H{\orv\'ath}}\ssp: {\em Topological Vector Spaces and
  Distributions\sp}, Addison--Wesley, Reading 1966.
\bibitem{Hunt}\bibname{B. R. H{\unt}, T. S{\auer} and J. A. Y{\orke}}\ssp: 
  `Prevalence: a translation-invariant \q{almost every} on 
  infinite-dimensional spaces' {\em Bull. Amer. Math. Soc\sp}. (N.S.) {\bf27} 
  2 (1992) 217\,--\,238.
\bibitem{Is}\bibname{R. I{\srael}}\ssp: `Is the set of entire functions Borel 
  in the space of analytic functions?' {\em Answers\sp} in 
  MathOverflow, http://mathoverflow.net/q/225792, 2015-12-10.
\bibitem{Ke}\bibname{H. H. K{\eller}}\ssp: {\em Differential Calculus in
  Locally Convex Spaces}\sp, Lecture Notes in Math. 417, Springer, Berlin -
  Heidelberg - New York  1974.
\bibitem{Ky}\bibname{J. L. K{\elley}}\ssp: {\em General Topology\sp},
  Graduate Texts in Math. 27, Springer, New York 1985.
\bibitem{KM}\bibname{A. K{\riegl} and P. W. M{\ichor}}\ssp: {\it The
  Convenient Setting of Global Analysis}\sp, Survey 53, Amer. Math. Soc.,
  Providence 1997.
\bibitem{Piz}\bibname{D. P{\izanelli}}\ssp: `Applications analytiques en 
  dimension infinie' {\em Bull. Sci. Math\sp}. {\bf96} 2 (1972) 181\,--\,191.

  \end{thebibliography}
\end{document}